\documentclass[11pt,leqno]{article}
\usepackage{amsmath,amssymb,amsthm}
\usepackage{mathrsfs}
\usepackage{titlesec,hyperref}

\usepackage{color}

\usepackage{latexsym,bm,graphicx}

\allowdisplaybreaks

\titleformat{\subsection}{\it}{\thesubsection.\enspace}{1pt}{}

\newtheorem{theo}{Theorem}[section]
\newtheorem{lemm}[theo]{Lemma}

\newtheorem{coro}[theo]{Corollary}

\newtheorem{rema}[theo]{Remark}

\begin{document}
\title{A Simple Method for the Optimal Transportation
\hspace{-4mm}
}

\author{Xian-Tao Huang$^1$
}
\footnotetext[1]{School of Mathematics and Computational Science, Sun Yat-sen University, Guangzhou, 510275, E-mail address: \it hxiant@mail2.sysu.edu.cn.}
\date{}
\maketitle

\begin{abstract}

In this paper we will give a new proof of the monotonicity of Wasserstein distances of two diffusions under super Ricci flow. Our proof is based on the coupling method of B. Andrew and J. Clutterbuck (see \cite{Ben}). The same method can also be applied to the contractivity of normalized $\mathscr{L}$-Wasserstein distance under backward Ricci flow.

\vspace*{5pt}
\noindent {\it 2000 Mathematics Subject Classification}: 53C44

\vspace*{5pt}
\noindent{\it Keywords}: Optimal transportation, super Ricci flow, Ricci flow.
\end{abstract}

\vspace*{10pt}

\section{Introduction}

Suppose $M$ is a compact oriented Riemannian manifold of dimension $n$. Let $\mu,\nu\in{P(M)}$ be two Borel probability measures on $M$. Let $c: M\times{M}\rightarrow\mathbb{R}\bigcup\{+\infty\}$ be a lower semi-continuous cost function. One can consider the Monge-Kantorovich minimization problem:
 \[\mathrm{T}_c(\mu,\nu)=\underset{\pi\in\Gamma(\mu,\nu)}{\inf}\int_{M\times{M}}c(x,y)d\pi(x,y),\]
where $\Gamma(\mu,\nu)$ is the set of Borel probability measures $\pi$ on $M\times{M}$ which have marginals $\mu$ and $\nu$ (i.e. $\pi(A\times{M})=\mu(A), \pi(M\times{A})=\nu(A)$ for every Borel set $A\subset{M}$).
In particular, when $c(x,y)=d^{p}(x,y)$, we have the $p$-Wasserstein distance $(p>0)$ between $\mu$ and $\nu$:
\[W_{p}(\mu,\nu)=\left(\underset{\pi\in\Gamma(\mu,\nu)}{\inf}\int_{M\times{M}}d^{p}(x,y)d\pi(x,y)\right)^\frac{1}{p}.\]

Many efforts have been devoted to characterize manifolds with lower bounds for the Ricci curvature (see e.g. \cite{Lott1} \cite{Sturm} \cite{Sturm4} and the references therein), and some of these characterizations use optimal transportation as a key tool.
For example, in \cite{Sturm}, Sturm and von Renesse proved that $Ric(M)\geq{K}$ is equivalent to the condition that $e^{Kt}W_{p}(\mu(t),\nu(t))$ is nonincreasing in $t$ for all $p\in[1,\infty]$, where $\mu(t),\nu(t)$ are two solutions of the heat equation.

The above papers all considered manifolds with static metrics. In the paper \cite{Topping2}, McCann and Topping first considered the equivalent properties of super Ricci flow, i.e. a smooth family of metrics $g_{\tau}$, $\tau\in[\tau_1,\tau_2]$, parameterized backward in time, on a compact oriented $n$-dimensional manifold $M$, satisfying
\begin{align}\label{superRF}
  -\frac{\partial}{\partial\tau}g_{\tau}+2Ric(g_{\tau})\geq0.
\end{align}

Let $\mu(\tau)$, $\tau\in(\tau_{1},\tau_{2})$, be a family of probability measures with $d\mu(\tau)=u(\cdot,\tau)dV_{\tau}$ such that $u$ satisfies the conjugate heat equation
\begin{align}\label{conjheateq}
  \frac{\partial{u}}{\partial\tau}=\Delta_{\tau}u-\left(\frac{1}{2}\textmd{tr}\frac{\partial{g}_{\tau}}{\partial\tau}\right)u,
\end{align}
where $dV_{\tau}$ and $\Delta_{\tau}$ denote the volume form and Laplacian with respect to $g_{\tau}$, respectively.
For brevity, we will refer to such a family $\mu(\tau)$ as a diffusion throughout this paper.
We will often abuse the notation of the probability $\mu$ and the volume form $d\mu$ for simplicity.

One can consider the problem of optimal transportation of two diffusions $d\mu(\tau)=u(\cdot,\tau)dV_\tau$, $d\nu(\tau)=v(\cdot,\tau)dV_\tau$.
Let $\eta:\mathbb{R}^+\times\mathbb{R}^+\rightarrow\mathbb{R}^+$ be a smooth function with $\eta(0,\tau)=0$. We consider a time-dependent cost function $c_{\tau}(x,y)=\eta(d_{\tau}(x,y),\tau)$, where $d_{\tau}(x,y)$ denotes the distance between $x$ and $y$ with respect to $g_{\tau}$.
The corresponding Monge-Kantorovich minimization problem is denoted by
\begin{align}
  \mathrm{T}_{c_{\tau}}(\mu(\tau),\nu(\tau))=\underset{\pi\in\Gamma(\mu(\tau),\nu(\tau))}{\inf}\int_{M\times{M}}c_{\tau}(x,y)d\pi(x,y),
\end{align}
The corresponding $p$-Wasserstein distance $(p>0)$ between $\mu(\tau)$ and $\nu(\tau)$ is denoted by $W_{p}(\mu(\tau),\nu(\tau)).$

In \cite{Topping2}, the authors proved the $2$-Wasserstein contractivity under super Ricci flow by calculating the derivatives of the entropy along Wasserstein geodesics. Lott \cite{Lott2} gave a new proof of the $2$-Wasserstein contractivity under Ricci flow. The $1$-Wasserstein contractivity was proved by Ilmanen, see \cite {Topping3}. Such monotonicity results were extended to a more general class of cost functions by Arnaudon, Coulibaly and Thalmaier \cite{Arnaudon} using a probabilistic method.

One of the purposes of this paper is to give an alternative proof of some of the monotonicity results in \cite{Arnaudon}. Our main theorem is as follows:

\begin{theo}\label{main}
If $g_{\tau}$ satisfies
\begin{align}\label{superKRF}
  -\frac{\partial}{\partial\tau}g_{\tau}+2Ric(g_{\tau})\geq2Kg_{\tau},
\end{align}
$d\mu(\tau)=u(x,\tau)dV_\tau$, $d\nu(\tau)=v(y,\tau)dV_\tau$ are two diffusions.
Furthermore, $\eta: \mathbb{R}^+\times\mathbb{R}^+\rightarrow\mathbb{R}^+$ satisfies
\begin{align}\label{condi_c}
   \left\{
  \begin{array}{l}
    \eta(0,\tau)=0,\\
    \frac{\partial}{\partial{s}}\eta(s,\tau)\geq0,\\ -\frac{\partial}{\partial\tau}\eta(s,\tau)+Ks\frac{\partial}{\partial{s}}\eta(s,\tau)-{\min\biggl{\{4\frac{\partial^{2}}{\partial{s}^{2}}\eta(s,\tau),0\biggr\}}}\geq0. \end{array}\right.
\end{align}
If $c_{\tau}(x,y)=\eta(d_{\tau}(x,y),\tau)$, then $\mathrm{T}_{c_{\tau}}(\mu(\tau),\nu(\tau))$ is nonincreasing in $\tau$.
\end{theo}

As a corollary, if $K=0$ and $\frac{\partial}{\partial\tau}\eta(s,\tau)=0$, we have

\begin{coro}\label{coro1}
Suppose $g_{\tau}$ satisfies $(\ref{superRF})$, $d\mu(\tau)$, $d\nu(\tau)$ are two diffusions, the function $\eta: \mathbb{R}^+\rightarrow\mathbb{R}^+$ satisfies $\eta(0)=0, \eta'\geq0$. If $c_{\tau}(x,y)=\eta(d_{\tau}(x,y))$, then $\mathrm{T}_{c_{\tau}}(\mu(\tau),\nu(\tau))$ is nonincreasing in $\tau$.
\end{coro}

For general $K$, if we choose $\eta(s,\tau)=e^{pK\tau}s^{p}$ for some ${p}>0$, then the conditions $(\ref{condi_c})$ are all satisfied. Notice that in this case,
\[e^{K\tau}W_{p}(\mu(\tau),\nu(\tau))=\{T_{c_{\tau}}(\mu(\tau),\nu(\tau))\}^{\frac{1}{p}},\]
we have

\begin{coro}\label{coro2}
If $g_{\tau}$, $d\mu(\tau)$, $d\nu(\tau)$ are the same as in Theorem \ref{main}, then $e^{K\tau}W_{p}(\mu(\tau),\nu(\tau))$ $({p}>0)$ is nonincreasing in $\tau$.
\end{coro}

\begin{rema}
If the family of metrics are fixed, one can recover the corresponding monotonicity of Wasserstein distances for manifolds with Ricci curvature bounded from below, which was proved by Sturm and von Renesse \cite{Sturm}.
\end{rema}

\begin{rema}
Corollary \ref{coro1} and Corollary \ref{coro2} were proved originally by a probabilistic method in \cite{Arnaudon}.
\end{rema}

The method we use originated from Andrew and Clutterbuck's papers \cite{Ben3} \cite{Ben}. They have used this method to bound the modulus of continuity of solutions of various parabolic equations. We will construct an operator that can be interpreted as a coupling of two Laplacians, and use the parabolic maximum principle in our argument.

There is an analogous notion of $\mathscr{L}$-Wasserstein distance introduced by Topping in \cite{Topping}.

Suppose we have the backward Ricci flow, $\frac{\partial}{\partial\tau}g_{\tau}=2Ric(g_{\tau})$, defined
on an open interval $I$ containing $[\bar{\tau}_{1},\bar{\tau}_{2}]$, where $0<\bar{\tau}_{1}<\bar{\tau}_{2}$. Perelman's $\mathscr{L}$-length of a smooth curve $\gamma:[\tau_{1},\tau_{2}]\rightarrow M$ $([\tau_{1},\tau_{2}]\subset{I})$ is defined to be
\[\mathscr{L}(\gamma):=\int_{\tau_{1}}^{\tau_{2}}\sqrt{\tau}(R(\gamma(\tau),\tau)+|\gamma'(\tau)|^{2}_{g_{\tau}})d\tau,\]
where $R(x,\tau)$ denotes the scalar curvature at $x$ in $(M,g_{\tau})$.
The $\mathscr{L}$-distance between a pair of points $(x,\tau_{1})$ and $(y,\tau_{2})$ is defined to be
\[Q(x,\tau_{1};y,\tau_{2}):=\inf\biggl\{\mathscr{L}(\gamma)|\gamma:[\tau_{1},\tau_{2}]\rightarrow M \text{ is smooth, } \gamma(\tau_{1})=x, \gamma(\tau_{2})=y\biggr\}.\]
One can also consider the notion of $\mathscr{L}$-Wasserstein distance $V(\nu_{1}(\tau_{1}),\tau_{1};\nu_{2}(\tau_{2}),\tau_{2})$ between
two diffusions $\nu_{1}(\tau)$ and $\nu_{2}(\tau)$:
\[V(\nu_{1}(\tau_{1}),\tau_{1};\nu_{2}(\tau_{2}),\tau_{2}):=\underset{\pi\in\Gamma(\nu_{1}(\tau_{1}),\nu_{2}(\tau_{2}))}{\inf}\int_{M\times{M}}Q(x,\tau_{1};y,\tau_{2})d\pi(x,y).\]

Now let $\tau_{1}=\tau_{1}(s):=\bar{\tau}_{1}e^{s}$, $\tau_{2}=\tau_{2}(s):=\bar{\tau}_{2}e^{s}$ be two exponential functions of $s\in\mathbb{R}$, and define the normalized distance between the diffusions $\nu_{1}(\tau_{1}(s))$ and $\nu_{2}(\tau_{2}(s))$ by
\[\Theta(s):=2(\sqrt{\tau_{1}}-\sqrt{\tau_{2}})V(\nu_{1}(\tau_{1}),\tau_{1};\nu_{2}(\tau_{2}),\tau_{2})-2n(\sqrt{\tau_{1}}-\sqrt{\tau_{2}})^{2}\]
for $s$ in a neighborhood of $0$ such that $\nu_{i}(\tau_{i}(s))$ are defined $(i=1,2)$.

Topping proved the following theorem in \cite{Topping}:

\begin{theo}[Topping]\label{Loptimal}
    $\Theta(s)$ is a nonincreasing function of $s$.
\end{theo}

Kuwada and Philipowski gave a new proof based on probabilistic methods in \cite{Kuwada}. We find that our new method also apply to this problem, see Section \ref{Lsec} for details.

\vspace*{2em}

\noindent\textbf{Acknowledgments}.The author would like to express his gratitude to his advisor Professor B.-L. Chen, who brought him this topic and gave him lots of enlightening discussions and encouragement. The author is very grateful to the referees for very careful reading and for critical comments to improve this paper.

\section{A coupling method}\label{couplingsec}

In this section, we will define a time-dependent operator $\mathscr{D}_{\tau}$ which can be viewed as a coupling of two Laplacians, and calculate the action of $\mathscr{D}_{\tau}$ on $\eta(d_\tau(x,y),\tau)$. The definition of $\mathscr{D}_{\tau}$ is inspired by the paper \cite{Ben}.

Suppose $\Omega$ is the set $(M\times{M})\backslash\{(x,x)|x\in{M}\}$. Denote the product metric $g_{\tau}\times{g_{\tau}}$ on $\Omega$ by $\tilde{g}_{\tau}$. Let $\tilde{\nabla}_{\tau}$, $\tilde{\nabla}_{\tau}^{2}$ be the gradient and Hessian with respect to $\tilde{g}_{\tau}$ respectively.

The set $\mathscr{F}_{\tau}\subset Sym_{2}(T^{*}\Omega)$ is defined to be
\[
\begin{split}
\mathscr{F}_{\tau}:=&\biggl\{A\in{Sym_{2}(T^{*}\Omega)}\biggl|A\geq0, A_{(x,y)}|_{T_{x}M\bigotimes{T_{x}M}}=g_{\tau}(x),\\
&A_{(x,y)}|_{T_{y}M\bigotimes{T_{y}M}}=g_{\tau}(y)\biggr\}.
\end{split}
\]
Obviously, $\tilde{g}_{\tau}\in\mathscr{F}_{\tau}$.
Suppose $(x,y)\in\Omega$, $\{E_{i}\}_{1\leq{i}\leq{n}}$ and $\{F_{i}\}_{1\leq{i}\leq{n}}$ are two orthonormal bases (with respect to $g_{\tau}$) defined in small open neighborhoods $U\ni{x}$ and $V\ni{y}$, respectively.
Denote $\{E^{i}_{\ast}\}_{1\leq{i}\leq{n}}$ and $\{F^{i}_{\ast}\}_{1\leq{i}\leq{n}}$ the dual coframes of $\{E_{i}\}_{1\leq{i}\leq{n}}$ and $\{F_{i}\}_{1\leq{i}\leq{n}}$, respectively.
Denote
$
\tilde{A}=\sum_{i=1}^{n}(E_{\ast}^{i},F_{\ast}^{i})\bigotimes(E_{\ast}^{i},F_{\ast}^{i})
$
on ${U}\times{V}$.
Suppose $U'$, $V'$ are open neighborhoods of $x$ and $y$, respectively, such that $\overline{U'}\subset{U}$, $\overline{V'}\subset{V}$.
Let $\alpha:\Omega\rightarrow[0,1]$ be a smooth cutoff function satisfying $\alpha(z)\equiv1$ on $U'\times{V'}$, $\alpha(z)\equiv0$ on $\Omega\setminus({U}\times{V})$.
Define $A_{1}$ to be \[A_{1}=\alpha\tilde{A}+(1-\alpha)\tilde{g}_{\tau},\]
it's easy to see $A_{1}\in\mathscr{F}_{\tau}$. Hence we can extend
$
\sum_{i=1}^{n}(E_{\ast}^{i},F_{\ast}^{i})\bigotimes(E_{\ast}^{i},F_{\ast}^{i})\biggl|_{(x,y)}
$
to an element of $\mathscr{F}_{\tau}$.

We define the operator $\mathscr{D}_{\tau}: C^{2}(\Omega)\rightarrow{C^{0}(\Omega)}$ to be

\[
\mathscr{D}_{\tau}(f(x,y)):=\inf\biggl\{{\textmd{tr}(A(\tilde{\nabla}_{\tau}^{2}f))}\biggl|A\in\mathscr{F}_{\tau}\biggr\}
\]
for ${f}\in{C^{2}(\Omega)}.$
When $f$ is independent of $x$, $\mathscr{D}_{\tau}(f)=\Delta_{\tau}f(y)$.
When $f$ is independent of $y$, $\mathscr{D}_{\tau}(f)=\Delta_{\tau}f(x)$.
Hence $\mathscr{D}_{\tau}$ is a coupling of two Laplacians $\Delta_{\tau}|_{x}$ and $\Delta_{\tau}|_{y}$.

Let $x,y\in{M}, x\neq{y}$, and $d=d_{\tau}(x,y)$. Let $ \gamma:[-\frac{d}{2},\frac{d}{2}]\rightarrow{M}$ be a minimizing geodesic from $x$ to $y$, parameterized by arc length in $(M,g_{\tau})$. Choose an orthonormal basis $\{E_{i}\}_{1\leq{i}\leq{n}}$ for $T_{x}M$, such that $E_{n}=\gamma'(-\frac{d}{2})$. Parallel transportation along $\gamma$ gives an orthonormal basis $\{E_{i}(s)\}_{1\leq{i}\leq{n}}$ for $T_{\gamma(s)}M$ with $E_{n}(s)=\gamma'(s)$.

From the definition of $\mathscr{D}_{\tau}$,
\begin{align}\label{Dineq}
\mathscr{D}_{\tau}f\leq&{\min}\biggl\{\sum\limits_{i=1}^{n}\tilde{\nabla}_{\tau}^{2}f((E_{i}(-\frac{d}{2}),E_{i}(\frac{d}{2})),(E_{i}(-\frac{d}{2}),E_{i}(\frac{d}{2}))),\\
&\sum\limits_{i=1}^{n-1}\tilde{\nabla}_{\tau}^{2}f((E_{i}(-\frac{d}{2}),E_{i}(\frac{d}{2})),(E_{i}(-\frac{d}{2}),E_{i}(\frac{d}{2})))\nonumber\\
&+\tilde{\nabla}_{\tau}^{2}f((E_{n}(-\frac{d}{2}),-E_{n}(\frac{d}{2})),(E_{n}(-\frac{d}{2}),-E_{n}(\frac{d}{2})))\biggr\},\nonumber
\end{align}
where the inequality holds in support sense.

Throughout this paper, the derivative of a Lipschitz function $\frac{d^{+}}{d\tau}f(\tau)$ is interpreted as
\[\frac{d^{+}}{d\tau}f(\tau)=\underset{h\downarrow0}{\limsup}\frac{f(\tau+h)-f(\tau)}{h}.\]
We will denote the $s, \tau$ derivatives of $\eta(s,\tau)$ by $\eta', \dot{\eta}$ respectively for brevity.
\begin{lemm}\label{lemmaaaa}
  Suppose $(M,g_{\tau})$ satisfies (\ref{superKRF}) and $\eta: \mathbb{R}^+\times\mathbb{R}^+\rightarrow\mathbb{R}^+$ satisfies $\eta'(s,\tau)\geq0$. If $c_{\tau}(x,y)=\eta(d_{\tau}(x,y),\tau)$, then
  \[
  (-\frac{d^{+}}{d\tau}-\mathscr{D}_{\tau})c_{\tau}(x,y)\geq-\dot{\eta}+K\eta'd_{\tau}-{\min}\{4\eta'',0\}
  \]
for $(x,y)\in\Omega$.
\end{lemm}
\begin{proof}

Let $\gamma:[-\frac{d}{2},\frac{d}{2}]\rightarrow{M}$ be any minimizing geodesic connecting $x$ and $y$ parameterized by arc length in $(M,g_{\tau})$. Let $\{E_{i}(s)\}_{1\leq{i}\leq{n}}$ be a parallel orthonormal basis for $T_{\gamma(s)}M$ with $E_{n}(s)=\gamma'(s)$.
For $i\in\{1,2,\ldots,n-1\}$, let $\gamma_{i}:(-\epsilon,\epsilon)\times[-\frac{d}{2},\frac{d}{2}]\rightarrow{M}$ be the variation $\gamma_{i}(r,s)=\exp_{\gamma(s)}(rE_{i})$, then  \[d_{\tau}(\exp_{x}(rE_{i}),\exp_{y}(rE_{i}))\leq{L_{\tau}[\gamma_{i}(r,\cdot)]},\]
with equality at $r=0$, where $L_{\tau}[\gamma]$ means the length of $\gamma$ with respect to $g_{\tau}$. Since $\eta(s,\tau)$ is nondecreasing in the first variable, we have
\[c_{\tau}(\exp_{x}(rE_{i}),\exp_{y}(rE_{i}))\leq{\eta}(L_{\tau}[\gamma_{i}(r,\cdot)],\tau),\]
with equality at $r=0$. By the first variation formula,
\[\frac{d}{dr}\biggl|_{r=0}L_{\tau}[\gamma_{i}(r,\cdot)]=0. \]
By the second variation formula,
    \[
    \begin{split}
    &\frac{d^{2}}{dr^{2}}\biggl|_{r=0}L_{\tau}[\gamma_{i}(r,\cdot)]\\
    =&\int^{\frac{d}{2}}_{-\frac{d}{2}}\{|\nabla_{\gamma'}E_{i}|^{2}-R(\gamma',E_{i},E_{i},\gamma')-(\frac{\partial}{\partial{s}}\langle{E_{i},\gamma'\rangle})^{2}\}ds+\langle\nabla_{E_{i}}E_{i},\gamma'\rangle\biggl|^{\frac{d}{2}}_{-\frac{d}{2}}\\
    =&-\int^{\frac{d}{2}}_{-\frac{d}{2}}R(\gamma',E_{i},E_{i},\gamma')ds.
    \end{split}
    \]
Therefore,
  \begin{align}\label{eq1}
  \sum^{n-1}_{i=1}\tilde{\nabla}_{\tau}^{2}c((E_{i},E_{i}),(E_{i},E_{i}))
  \leq&\sum^{n-1}_{i=1}\frac{d^{2}}{dr^{2}}\biggl|_{r=0}\eta(L_{\tau}[\gamma_{i}(r,\cdot)],\tau)\\
  =&\sum^{n-1}_{i=1}\biggl[\eta''\biggl(\frac{d}{dr}\biggl|_{r=0}L_{\tau}\biggl)^{2}+\eta'\frac{d^{2}}{dr^{2}}\biggl|_{r=0}L_{\tau}\biggl]\nonumber\\
  =&-\eta'\int^{\frac{d}{2}}_{-\frac{d}{2}}Ric(\gamma',\gamma')ds.\nonumber
  \end{align}
Similarly, if we extend $\gamma$ a little to get a longer geodesic and we will still denote it by $\gamma:[-\frac{d}{2},\frac{d}{2}+\delta]\rightarrow{M}$, then
\[
  d_{\tau}(\gamma(-\frac{d}{2}+r),\gamma(\frac{d}{2}+r))
  \leq{L_{\tau}[\gamma|_{(-\frac{d}{2}+r,\frac{d}{2}+r)}]}\equiv{d},
\]
with equality at $r=0$.
Therefore,
\[c_{\tau}(\gamma(-\frac{d}{2}+r),\gamma(\frac{d}{2}+r))\leq{\eta}(d,\tau).\]
Hence
  \begin{align}\label{eq2}
    \tilde{\nabla}_{\tau}^{2}c((E_{n},E_{n}),(E_{n},E_{n}))\leq0.
  \end{align}
It is easy to see that
  \begin{align}\label{eq3}
  &\tilde{\nabla}_{\tau}^{2}c((E_{n},-E_{n}),(E_{n},-E_{n}))\\
  =&\frac{d^{2}}{dr^{2}}\biggl|_{r=0}\eta(d-2r,\tau)=4\eta''.\nonumber
  \end{align}
Combining $(\ref{Dineq}), (\ref{eq1}), (\ref{eq2}), (\ref{eq3})$, we get
\begin{align}\label{equ2}
  \mathscr{D}_{\tau}c_{\tau}(x,y)\leq{-}\eta'\int^{\frac{d}{2}}_{-\frac{d}{2}}Ric(\gamma',\gamma')ds+\min\{4\eta'',0\}.
\end{align}
On the other hand, we get
\begin{align}
\frac{d^{+}}{d\tau}d_{\tau}(x,y)&=\underset{h\downarrow0}{\limsup}\frac{d_{\tau+h}(x,y)-d_{\tau}(x,y)}{h}\\
&\leq\underset{h\downarrow0}{\limsup}\frac{L_{\tau+h}[\gamma]-L_{\tau}[\gamma]}{h}\nonumber\\
&=\frac{d}{d\tau}L_{\tau}[\gamma]=\frac{1}{2}{\int^{\frac{d}{2}}_{-\frac{d}{2}}}{\frac{\partial}{\partial\tau}g(\gamma',\gamma')}ds.\nonumber
\end{align}
Therefore,
\begin{align}\label{equ1}
  -\frac{d^{+}}{d\tau}c_{\tau}(x,y)&=\eta'{\biggl(-\frac{d^{+}}{d\tau}d_{\tau}(x,y)\biggl)}-\dot{\eta}\\
   &\geq\frac{\eta'}{2}\biggl({{\int^{\frac{d}{2}}_{-\frac{d}{2}}}{-\frac{\partial}{\partial\tau}g(\gamma',\gamma')}ds}\biggl)-\dot{\eta}.\nonumber
\end{align}
Combining $(\ref{equ2}), (\ref{equ1})$, we get
  \[
  \begin{split}
  &(-\frac{d^{+}}{d\tau}-\mathscr{D}_{\tau})c_{\tau}(x,y)\\
  \geq&\frac{\eta'}{2}{\int^{\frac{d}{2}}_{-\frac{d}{2}}}\biggl(-\frac{\partial}{\partial\tau}g+2Ric\biggl)(\gamma',\gamma')ds-\dot{\eta}-\min\{4\eta'',0\}\\
  \geq&{Kd\eta'-\dot{\eta}-\min\{4\eta'',0\}},
  \end{split}
  \]
where we have used (\ref{superKRF}) in the last inequality.
So we have completed the proof.
\end{proof}

\section{Proof of Theorem \ref{main}}

The following lemma of Kantorovich duality is well known to the experts.

\begin{lemm}[Kantorovich duality, see \cite{Villani}]\label{dual}
  Suppose $M$ is a manifold. Let $\mu,\nu\in{P(M)}$, and $c: M\times{M}\rightarrow\mathbb{R}\bigcup\{+\infty\}$ be a lower semi-continuous cost function. Define
\begin{align}
  J(\varphi,\psi)=\int_{M}\varphi{d\mu}+\int_{M}\psi{d\nu}
\end{align}
for $(\varphi,\psi)\in{L}^{1}(d\mu)\times{L}^{1}(d\nu)$.
Let $\Phi_{c}$ be the set of all $(\varphi,\psi)\in{L}^{1}(d\mu)\times{L}^{1}(d\nu)$ satisfying
\begin{align}
  \phi(x)+\psi(y)\leq{c(x,y)}
\end{align}
for $d\mu$-almost all $x\in{M}$, $d\nu$-almost all $y\in{M}$. Then
\begin{align}
    \mathrm{T}_c(\mu,\nu)=\underset{(\varphi,\psi)\in\Phi_{c}}{\sup}J(\varphi,\psi).
\end{align}
\end{lemm}

A pair of functions $(\varphi,\psi)\in{\Phi_{c}}$ is said to be competitive.

Now we begin the proof of Theorem \ref{main}, using Lemma \ref{dual} together with the ideas of Section 3 in \cite{Ben}.
\begin{proof}[Proof of Theorem \ref{main}]

Suppose $b\in(\tau_{1},\tau_{2})$. For any $\varepsilon>0$, there exists $(\alpha_{b},\beta_{b})$ such that $\alpha_{b}(x)+\beta_{b}(y)\leq{c_{b}(x,y)}$ and $J_{b}(\alpha_{b},\beta_{b})>T_{c_{b}}(\mu,\nu)-\epsilon$. We solve the equations
\begin{align}
  \left\{
     \begin{array}{ll}
       -\frac{\partial\varphi}{\partial\tau}=\Delta_{\tau}\varphi,&\tau\in(\tau_{1},b)\\
        \varphi(x,b)=\alpha_{b}(x),
     \end{array}
   \right.
\end{align}
\begin{align}
  \left\{
     \begin{array}{ll}
       -\frac{\partial\psi}{\partial\tau}=\Delta_{\tau}\psi,&\tau\in(\tau_{1},b)\\
        \psi(x,b)=\beta_{b}(x).
     \end{array}
   \right.
\end{align}
Since
\[
\frac{d}{d\tau}\int_{M}\varphi{u}dV_{\tau}=\int_{M}\biggl[\frac{\partial\varphi}{\partial\tau}u+\varphi(\frac{\partial{u}}{\partial\tau})+\frac{1}{2}\textmd{tr}(\frac{\partial{g}}{\partial\tau})\varphi{u}\biggr]dV_{\tau}\\
=0
\]
and $\frac{d}{d\tau}\int_{M}\psi{v}dV_{\tau}=0$, we have
\[\frac{d}{d\tau}J_{\tau}(\varphi(\cdot,\tau),\psi(\cdot,\tau))=0.\]
If we can prove $\varphi(x,\tau)+\psi(y,\tau)\leq{c}_{\tau}(x,y)$ for every $\tau\in(\tau_{1},b)$, then
\[
\begin{split}
T_{c_{b}}(\mu,\nu)\leq&{J_{b}(\alpha_{b},\beta_{b})+\epsilon}\\
=&J_{\tau}(\varphi_{\tau},\psi_{\tau})+\epsilon\\
\leq&T_{c_{\tau}}(\mu,\nu)+\epsilon.
\end{split}
\]
By the arbitrariness of $\epsilon$, we have $T_{c_{b}}(\mu,\nu)\leq{T}_{c_{\tau}}(\mu,\nu), \forall\tau\in(\tau_{1},b),$
i.e. $T_{c_{\tau}}(\mu,\nu)$ is nonincreasing in $\tau$.

In the following, we will prove $\varphi(x,\tau)+\psi(y,\tau)\leq{c}_{\tau}(x,y)=\eta(d_{\tau}(x,y),\tau)$ for $\tau\in(\tau_{1},b)$.
Define an evolving quantity $Z$ on $M\times{M}\times(\tau_{1},b)$: \[Z(x,y,\tau)=\varphi(x,\tau)+\psi(y,\tau)-{c}_{\tau}(x,y)+\epsilon(\tau-b-1).\]

Notice that since $\eta(s,\tau)$ satisfies $(\ref{condi_c})$, we have $(-\frac{d^{+}}{d\tau}-\mathscr{D}_{\tau}){c}_{\tau}(x,y)\geq0$ by Lemma \ref{lemmaaaa}.

We now apply the operator $-\frac{d^{+}}{d\tau}-\mathscr{D}_{\tau}$ to $Z$ :
\[
\begin{split}
&(-\frac{d^{+}}{d\tau}-\mathscr{D}_{\tau})Z(x,y,\tau)\\
=&-\frac{\partial}{\partial\tau}\varphi(x,\tau)-\frac{\partial}{\partial\tau}\psi(y,\tau)+\frac{d^{+}}{d\tau}{c}_{\tau}(x,y)-\epsilon\\
&-\Delta_{\tau}\varphi(x,\tau)-\Delta_{\tau}\psi(y,\tau)+\mathscr{D}_{\tau}{c}_{\tau}(x,y)\\
=&-\epsilon+(\frac{d^{+}}{d\tau}+\mathscr{D}_{\tau}){c}_{\tau}(x,y)\\
\leq&-\epsilon<0.
\end{split}
\]

By assumption, we have $Z(x,y,b)\leq-\epsilon<0$. On the other hand, let $f(x,\tau)=\varphi(x,\tau)+\psi(x,\tau)$, then $f$ satisfies $-\frac{\partial{f}}{\partial\tau}=\Delta_{\tau}f$, with $f(x,b)=\varphi(x,b)+\psi(x,b)\leq{\eta(0,b)}=0$. By maximum principle, we have $f(x,\tau)\leq0=\eta(0,\tau), \forall\tau\in(\tau_{1},b)$. So $Z(x,x,\tau)<0, \forall\tau\in(\tau_{1},b)$. It follows that if $Z$ ever becomes positive, then there exists a maximal $\tau_{0}<b$ and $x_{0}\neq{y_{0}}$ in $M$ such that $Z(x_{0},y_{0},\tau_{0})=0$. Then at $(x_{0},y_{0},\tau_{0})$,
\[\frac{d^{+}}{d\tau}Z\leq0,\qquad\mathscr{D}_{\tau}Z\leq0.\]
Therefore, $(-\frac{d^{+}}{d\tau}-\mathscr{D}_{\tau})Z(x_{0},y_{0},\tau_{0})\geq0$. We get a contradiction. Hence $Z(x,y,\tau)\leq0$. By the arbitrariness of $\epsilon$, we get $\varphi(x,\tau)+\psi(y,\tau)\leq{c}_{\tau}(x,y)$ and finish the proof.
\end{proof}

\section{L-optimal transportation}\label{Lsec}

Before starting our new proof of Theorem \ref{Loptimal}, we recall some basic theory of Perelman's $\mathscr{L}$-length. The readers can refer to \cite{CaoZhu} \cite{Chow} \cite{Kleiner} \cite{Perelman} for more details and further results.
Suppose we have the backward Ricci flow, $\frac{\partial}{\partial\tau}g_{\tau}=2Ric(g_{\tau})$.
For a curve $\gamma:[\tau_{1},\tau_{2}]\rightarrow M$ with $\tau_{2}>\tau_{1}>0$, denote $X(\tau)=\gamma'(\tau)$, and let $Y(\tau)$ be a smooth vector field along $\gamma(\tau)$. The first variation formula of $\mathscr{L}$-length is:
\[
  \delta_{Y}[\mathscr{L}]=2\sqrt{\tau}\langle X,Y\rangle\biggl|_{\tau_{1}}^{\tau_{2}}+\int_{\tau_{1}}^{\tau_{2}}\sqrt{\tau}\langle Y,\nabla R-2\nabla_{X}X-4Ric(\cdot,X)-\frac{1}{\tau}X\rangle d\tau,
\]
where $\langle\cdot,\cdot\rangle$ denotes the inner product with respect to $g_{\tau}$.

A smooth curve $\gamma(\tau)$ in M is called an $\mathscr{L}$-geodesic if it satisfies the following $\mathscr{L}$-geodesic equation:
\begin{align}\label{L-geo-eq}
  2\nabla_{X}X-\nabla R+4Ric(\cdot,X)+\frac{1}{\tau}X=0.
\end{align}

Along an $\mathscr{L}$-geodesic $\gamma$, we have the second variation formula
\begin{align}
  \delta^{2}_{Y}[\mathscr{L}]=2\sqrt{\tau}\langle \nabla_{Y}Y,X\rangle\biggl|_{\tau_{1}}^{\tau_{2}}+\int_{\tau_{1}}^{\tau_{2}}\sqrt{\tau}\biggl[2|\nabla_{X}Y|^{2}+2\langle R(Y,X)Y,X\rangle\\+\textmd{Hess}R(Y,Y)
  +2\nabla_{X}Ric(Y,Y)-4\nabla_{Y}Ric(Y,X)\biggr]d\tau.\nonumber
\end{align}

Suppose $\gamma:[\tau_{1},\tau_{2}]\rightarrow M$ is a shortest $\mathscr{L}$-geodesic connecting $(x,\tau_{1})$ and $(y,\tau_{2})$.
For $(\hat{x},\hat{\tau}_{1})$, $(\hat{y},\hat{\tau}_{2})$ near $(x,\tau_{1})$ and $(y,\tau_{2})$ respectively,
denote $\mathscr{L}(\hat{x},\hat{\tau}_{1};\hat{y},\hat{\tau}_{2})$ the $\mathscr{L}$-length of the $\mathscr{L}$-geodesic $\gamma_{(\hat{x},\hat{\tau}_{1};\hat{y},\hat{\tau}_{2})}$ connecting $(\hat{x},\hat{\tau}_{1})$ and $(\hat{y},\hat{\tau}_{2})$ near $\gamma$, then
$\mathscr{L}(\hat{x},\hat{\tau}_{1};\hat{y},\hat{\tau}_{2})\geq{Q}(\hat{x},\hat{\tau}_{1};\hat{y},\hat{\tau}_{2})$, with equality when $(\hat{x},\hat{\tau}_{1};\hat{y},\hat{\tau}_{2})=(x,\tau_{1};y,\tau_{2})$.
By the computations similar to Perelman's \cite{Perelman} (see also Lemma A.6 in \cite{Topping}), we can derive
\begin{align}\label{partL}
  &\tau_{1}\frac{\partial}{\partial\tau_{1}}\mathscr{L}(x,\tau_{1};y,\tau_{2})+\tau_{2}\frac{\partial}{\partial\tau_{2}}\mathscr{L}(x,\tau_{1};y,\tau_{2})\\
  =&2\tau_{2}^{\frac{3}{2}}R(y,\tau_{2})-2\tau_{1}^{\frac{3}{2}}R(x,\tau_{1})+\mathscr{K}-\frac{1}{2}\mathscr{L}(x,\tau_{1};y,\tau_{2}),\nonumber
\end{align}
where $\mathscr{K}:=\int_{\tau_{1}}^{\tau_{2}}\tau^{\frac{3}{2}}H(X(\tau))d\tau$, and $H(X)$ is Hamilton's trace Harnack quantity (with $t=-\tau$)
\[H(X):=-\frac{\partial{R}}{\partial\tau}-\frac{1}{\tau}R-2\langle\nabla R,X\rangle+2Ric(X,X).\]

Now we define an operator $\mathscr{D}$ which is a coupling of $\tau_{1}\Delta_{\tau_{1}}|_{x}$ and $\tau_{2}\Delta_{\tau_{2}}|_{y}$ as in Section \ref{couplingsec}:
\[
\begin{split}
\mathscr{D}(f(x,\tau_{1};y,\tau_{2}))&:=\inf\biggl\{{\textmd{tr}(A(\tilde{\nabla}^{2}f))}\biggl| A\in{Sym_{2}(T^{*}(M\times{M}))},A\geq0,\\ &A_{(x,y)}|_{T_{x}M\bigotimes{T_{x}M}}=\tau_{1}g_{\tau_{1}}(x),
A_{(x,y)}|_{T_{y}M\bigotimes{T_{y}M}}=\tau_{2}g_{\tau_{2}}(y)\biggr\}
\end{split}
\]
for $f: M\times{I}\times{M}\times{I}\rightarrow\mathbb{R}$, where $\tilde{\nabla}^{2}$ means the Hessian with respect to the product metric of $g_{\tau_{1}}$ and $g_{\tau_{2}}$.

Suppose $\gamma:[\tau_{1},\tau_{2}]\rightarrow M$ is a shortest $\mathscr{L}$-geodesic connecting $(x,\tau_{1})$ and $(y,\tau_{2})$.
Let $\{Y_{i}\}_{1\leq i\leq n}$ be a basis at $\gamma(\tau_{1})$ with $\langle Y_{i},Y_{j}\rangle_{g_{\tau_{1}}}=\tau_{1}\delta_{ij}$. We extend this basis along $\gamma$ to get a family of bases $\{Y_{i}(\tau)\}_{1\leq i\leq n}$ by solving the ODEs
\[\nabla_{X}Y_{i}=-Ric(Y_{i},\cdot)+\frac{1}{2\tau}Y_{i}.\]
From
\[\frac{d}{d\tau}\langle Y_{i},Y_{j}\rangle=2Ric(Y_{i},Y_{j})+\langle \nabla_{X}Y_{i},Y_{j}\rangle+\langle Y_{i},\nabla_{X}Y_{j}\rangle=\frac{1}{\tau}\langle Y_{i},Y_{j}\rangle,\]
it follows
\[\langle Y_{i}(\tau),Y_{j}(\tau)\rangle_{g_{\tau}}=\tau\delta_{ij}.\]
By the definition of $\mathscr{D}$, we have
\[\mathscr{D}f\leq\sum_{i=1}^{n}\tilde{\nabla}^{2}f((Y_{i}(\tau_{1}),Y_{i}(\tau_{2})),(Y_{i}(\tau_{1}),Y_{i}(\tau_{2}))),\]
where the above inequality holds in support sense.

For any $i\in\{1,\ldots,n\}$,
let $\gamma_{i}:(-\epsilon,\epsilon)\times[\tau_{1},\tau_{2}]\rightarrow{M}$, $\gamma_{i}(r,s)=\exp_{\gamma(s)}(rY_{i})$ be a variation of $\gamma$, then  \[Q(\exp_{x}(rY_{i}),\tau_{1};\exp_{y}(rY_{i}),\tau_{2})\leq{\mathscr{L}[\gamma_{i}(r,\cdot)]},\]
with equality at $r=0$. Hence

\begin{align}
  &\tilde{\nabla}^{2}Q((Y_{i}(\tau_{1}),Y_{i}(\tau_{2})),(Y_{i}(\tau_{1}),Y_{i}(\tau_{2})))\nonumber\\
    \leq&\frac{d^{2}}{dr^{2}}\biggl|_{r=0}\mathscr{L}[\gamma_{i}(r,\cdot)]\nonumber\\
    =&\int_{\tau_{1}}^{\tau_{2}}\sqrt{\tau}\biggl[2|\nabla_{X}Y_{i}|^{2}+2\langle R(Y_{i},X)Y_{i},X\rangle+\textmd{Hess}R(Y_{i},Y_{i})+2\nabla_{X}Ric(Y_{i},Y_{i})\nonumber\\
    &-4\nabla_{Y_{i}}Ric(Y_{i},X)\biggr]d\tau\nonumber\\
    =&-2\sqrt{\tau}Ric(Y_{i},Y_{i})\biggl|_{\tau_{1}}^{\tau_{2}}+\sqrt{\tau}\biggl|_{\tau_{1}}^{\tau_{2}}-\int_{\tau_{1}}^{\tau_{2}}\sqrt{\tau}\biggl[-2\langle R(Y_{i},X)Y_{i},X\rangle-\textmd{Hess}R(Y_{i},Y_{i})\nonumber\\
  &-4\nabla_{X}Ric(Y_{i},Y_{i})+4\nabla_{Y_{i}}Ric(X,Y_{i})-\frac{1}{\tau}Ric(Y_{i},Y_{i})-2\frac{\partial}{\partial\tau}Ric(Y_{i},Y_{i})+2|Ric(Y_{i},\cdot)|^{2}\biggl]d\tau,\nonumber
\end{align}
where in the last equality we have left out the well known computations due to Perelman \cite{Perelman}. The readers can also consult \cite{CaoZhu}, \cite{Chow}, \cite{Kleiner} for details of the computations.

Summing over $i$, we get
\begin{align}\label{llllll}
  \mathscr{D}Q
  \leq&n\sqrt{\tau}\biggl|_{\tau_{1}}^{\tau_{2}}-2\tau^{\frac{3}{2}}R\biggl|_{\tau_{1}}^{\tau_{2}}-\int_{\tau_{1}}^{\tau_{2}}\tau^{\frac{3}{2}}H(X)d\tau\\
  =&n(\sqrt{\tau_{2}}-\sqrt{\tau_{1}})-(2\tau_{2}^{\frac{3}{2}}R(y,\tau_{2})-2\tau_{1}^{\frac{3}{2}}R(x,\tau_{1}))-\mathscr{K}.\nonumber
\end{align}

Now we begin our new proof of Theorem \ref{Loptimal}.
\begin{proof}[Proof of Theorem \ref{Loptimal}]
Suppose the backward Ricci flow is defined on an open interval $I$ containing $[\bar{\tau}_{1},\bar{\tau}_{2}]$, where $0<\bar{\tau}_{1}<\bar{\tau}_{2}$.
Let $\tau_{1}=\tau_{1}(s):=\bar{\tau}_{1}e^{s}$, $\tau_{2}=\tau_{2}(s):=\bar{\tau}_{2}e^{s}$, and \[P(x,y,s):=2(\sqrt{\tau_{2}}-\sqrt{\tau_{1}})Q(x,\tau_{1};y,\tau_{2})-2n(\sqrt{\tau_{2}}-\sqrt{\tau_{1}})^{2}.\]
If $\gamma(\tau)$ is a shortest $\mathscr{L}$-geodesic connecting $(x,\tau_{1})$ and $(y,\tau_{2})$, then by $(\ref{llllll})$,
\begin{align}\label{Lcase1}
  \mathscr{D}P(x,y,s)\leq2(\sqrt{\tau_{2}}-\sqrt{\tau_{1}})[n(\sqrt{\tau_{2}}-\sqrt{\tau_{1}})-(2\tau_{2}^{\frac{3}{2}}R(y,\tau_{2})-2\tau_{1}^{\frac{3}{2}}R(x,\tau_{1}))-\mathscr{K}].
\end{align}

On the other hand, by $(\ref{partL})$,
\[
\begin{split}
\frac{d^{+}}{ds}Q(x,\tau_{1};y,\tau_{2})&:=\underset{h\downarrow0}{\limsup}\frac{Q(x,\tau_{1}(s+h);y,\tau_{2}(s+h))-Q(x,\tau_{1}(s);y,\tau_{2}(s))}{h}\\
&\leq\frac{d}{ds}\mathscr{L}(x,\tau_{1}(s);y,\tau_{2}(s))\\
&=\tau_{1}\frac{\partial}{\partial\tau_{1}}\mathscr{L}(x,\tau_{1};y,\tau_{2})+\tau_{2}\frac{\partial}{\partial\tau_{2}}\mathscr{L}(x,\tau_{1};y,\tau_{2})\\
&=2\tau_{2}^{\frac{3}{2}}R(y,\tau_{2})-2\tau_{1}^{\frac{3}{2}}R(x,\tau_{1})+\mathscr{K}-\frac{1}{2}Q(x,\tau_{1};y,\tau_{2}).
\end{split}
\]
Therefore, we have
\begin{align}\label{Lcase2}
  &\frac{d^{+}}{ds}P(x,y,s)\\
  =&2(\sqrt{\tau_{2}}-\sqrt{\tau_{1}})(\frac{d^{+}}{ds}Q(x,\tau_{1};y,\tau_{2}))\nonumber\\
  &+(\sqrt{\tau_{2}}-\sqrt{\tau_{1}})Q(x,\tau_{1};y,\tau_{2})-2n(\sqrt{\tau_{2}}-\sqrt{\tau_{1}})^{2}\nonumber\\
  \leq&2(\sqrt{\tau_{2}}-\sqrt{\tau_{1}})[2\tau_{2}^{\frac{3}{2}}R(y,\tau_{2})-2\tau_{1}^{\frac{3}{2}}R(x,\tau_{1})+\mathscr{K}-n(\sqrt{\tau_{2}}-\sqrt{\tau_{1}})].\nonumber
\end{align}
Combining $(\ref{Lcase1}), (\ref{Lcase2})$, we get
\begin{align}\label{P-inequa}
(-\frac{d^{+}}{ds}-\mathscr{D})P(x,y,s)\geq0.
\end{align}
It is easy to see
\[
\Theta(s)=\underset{\pi\in\Gamma(\nu_{1}(\tau_{1}),\nu_{2}(\tau_{2}))}{\inf}\int_{M\times{M}}P(x,y,s)d\pi(x,y).
\]

Now we can use Lemma \ref{dual} and the maximum principle to prove the theorem just as what we have done in the proof of Theorem \ref{main}. The only difference is that the price functions $\varphi(s,x)$ and $\psi(s,y)$ solve the equation
\[-\frac{\partial}{\partial s}f=\tau(s)\Delta_{\tau(s)}f,\]
which is conjugate to the equation
\[\frac{\partial}{\partial s}u=\tau(s)[\Delta_{\tau(s)}u+R(x,\tau(s))u].\]
Finally, $(\varphi(s,x),\psi(s,y))$ will remain competitive because of $(\ref{P-inequa})$ together with the maximum principle.
\end{proof}

\phantomsection

\end{document}